\documentclass[11pt]{article}

\usepackage{graphicx} 
\usepackage[a4paper,left=3cm,right=3cm,top=3cm,bottom=3cm]{geometry}
\usepackage{amssymb, amsmath}
\usepackage{mathtools}
\usepackage{enumitem}
\usepackage{amsthm}
\usepackage[colorlinks=true,
    linkcolor=blue,
    filecolor=magenta,      
    urlcolor=blue, citecolor=red]{hyperref}
\usepackage[capitalize]{cleveref}
\usepackage{caption}
\usepackage{subcaption}
\usepackage{authblk}

\usepackage{microtype} 

\makeatletter
\def\cref@thmoptarg[#1]#2#3#4{%
    \ifhmode\unskip\unskip\par\fi%
    \normalfont%
    \trivlist%
    \let\thmheadnl\relax%
    \let\thm@swap\@gobble%
    \thm@notefont{\fontseries\mddefault\upshape}%
    \thm@headpunct{.}
    \thm@headsep 5\p@ plus\p@ minus\p@\relax%
    \thm@space@setup%
    #2
    \@topsep \thm@preskip               
    \@topsepadd \thm@postskip           
    \def\@tempa{#3}\ifx\@empty\@tempa%
      \def\@tempa{\@oparg{\@begintheorem{#4}{}}[]}%
    \else%
      \refstepcounter[#1]{#3}
      \@namedef{cref@#3@alias}{#1}
      \def\@tempa{\@oparg{\@begintheorem{#4}{\csname the#3\endcsname}}[]}%
    \fi%
    \@tempa}%
\makeatother

\newtheorem{theorem}{Theorem}
\newtheorem{lemma}[theorem]{Lemma}
\newtheorem{corollary}[theorem]{Corollary}
\newtheorem{proposition}[theorem]{Proposition}
\newtheorem{problem}[theorem]{Problem}

\theoremstyle{definition}

\newtheorem{example}[theorem]{Example}
\newtheorem{remark}[theorem]{Remark}

\usepackage[square,numbers]{natbib}
\usepackage{doi}

\usepackage{array}
\newcolumntype{C}[1]{>{\centering\let\newline\\\arraybackslash\hspace{0pt}}m{#1}}
\usepackage{multirow}

\renewcommand{\epsilon}{\varepsilon}

\newcommand{\Mat}{\mathrm{Mat}}
\usepackage{dsfont}

\newcommand{\be}{\begin{eqnarray}}
\newcommand{\ee}{\end{eqnarray}}
\newcommand{\ben}{\begin{enumerate}}
\newcommand{\een}{\end{enumerate}}

\newcommand{\ba}{\begin{array}}
\newcommand{\ea}{\end{array}}

\newcommand{\tr}{\mathrm{tr}}

\newcommand{\NP}{\mathsf{NP}}

\makeatletter
\def\p@subsection{}
\def\p@subsubsection{}
\makeatother
\usepackage[textsize=footnotesize]{todonotes}

\let\originalleft\left
\let\originalright\right
\renewcommand{\left}{\mathopen{}\mathclose\bgroup\originalleft}
\renewcommand{\right}{\aftergroup\egroup\originalright}
\usepackage{fancyhdr} 

\newcommand\blfootnote[1]{%
  \begingroup
  \renewcommand\thefootnote{}\footnote{#1}%
  \addtocounter{footnote}{-1}%
  \endgroup
}

\providecommand{\subjclass}[1]{\blfootnote{\textit{2020 Mathematics Subject Classification.} #1}}

\begin{document}

\title{Positive Moments Forever:\\ Undecidable and Decidable Cases}
\author[1]{Gemma De les Coves}
\author[1]{Joshua Graf}
\author[1,2]{Andreas Klingler\thanks{Corresponding author. E-mail: \href{mailto:andreas.klingler@univie.ac.at}{andreas.klingler@univie.ac.at}}}
\author[3]{Tim Netzer}

\renewcommand\Affilfont{\itshape\small}
\affil[1]{Institute for Theoretical Physics, University of Innsbruck, Austria}
\affil[2]{Faculty of Mathematics, University of Vienna, Austria}
\affil[3]{Department of Mathematics, University of Innsbruck, Austria}

\date{\today}

\pagestyle{fancy}
\renewcommand{\headrulewidth}{0.1pt}
\fancyhf{}
\makeatletter
\newcommand{\runtitle}{Moment problems}
\makeatother
\fancyhead[OL]{\nouppercase\leftmark}
\fancyhead[ER]{\textsc\runtitle}
\fancyhead[EL,OR]{\thepage}
\maketitle

\begin{abstract}
We investigate the generalized moment membership problem for matrices, a formulation equivalent to Skolem’s problem for linear recurrence sequences. We show decidability for orthogonal, unitary, and real eigenvalue matrices, and undecidability for matrices over certain commutative and non-commutative polynomial rings. As consequences, we deduce that positivity is decidable for simple unitary linear recurrence sequences and undecidable for linear recurrence sequences over commutative polynomial rings. As a byproduct, we also prove a free version of Pólya’s theorem.
\end{abstract}



\section{Introduction}
Some problems may look innocent yet be formally very difficult\,---\,perhaps uncomputable\,---\,or even worse, their computability may be unknown. 
\emph{Skolem's problem} exemplifies this uncertainty, focusing on the behavior of \emph{linear recurrence sequences} (LRS), where each term in the sequence is generated linearly from its predecessors. Examples of LRS include well-known sequences like the Fibonacci sequence or those derived from discretizing differential equations. Despite their simplicity, LRS are fundamental in various mathematical and computer science domains, notably in generating pseudo-random numbers \cite{Ta65}, describing the dynamics of cellular automata \cite{Ma84}, and many other applications \cite{Ev03}.

Specifically, an LRS of order $s$ is given by
$$
u_n = a_1 u_{n-1} + a_2 u_{n-2} + \cdots + a_s u_{n-s}
$$
where $a_1, \ldots, a_s \in \mathcal{R}$ are fixed elements in a ring $\mathcal{R}$, usually commutative. Together with initial values $u_1, \ldots, u_s \in \mathcal{R}$, this gives rise to a sequence $(u_n)_{n \in \mathbb{N}}$ in $\mathcal{R}$. While several important examples of LRS are over the ring $\mathcal{R} = \mathbb{Z}$, many interesting examples are defined over other rings. For example, the Chebyshev polynomials are defined via the LRS
$$ T_{n}(x) = 2x T_{n-1}(x) - T_{n-2}(x) \quad \text{ with } T_1(x) \coloneqq x \text{ and } T_0(x) \coloneqq 1$$
over the commutative ring  $\mathbb{Z}[x]$ of univariate polynomials.

Skolem's problem is a long-standing open question concerning LRS over $\mathbb{Z}$ \cite{Ou12}. It asks whether an algorithm exists that decides if an LRS attains the value $0$ for some $n \in \mathbb{N}$.
While partial solutions to Skolem's problem are known, implying decidability for order $s \leqslant 4$ \cite{Ti84, Ve85}, they do not apply to recurrences of order five or more.
A modification of Skolem's problem is the \emph{positivity problem} for LRS. Instead of asking whether the LRS is non-zero, it asks whether it stays non-negative.
In this case it is also unclear whether an algorithm exists that decides the positivity problem, as decidability is proven only for $s \leqslant 5$ \cite{OW, OW2}.

Examples for LRS are \emph{moment sequences}, in which we have 
$$ 
u_n = \tr(A^n),
$$
or \emph{generalized moment sequences}, in which $$u_n=\varphi(A^n)$$ 
for a given matrix $A \in \mathrm{Mat}_s(\mathcal{R})$ and a linear functional  $\varphi$ on $\mathrm{Mat}_s(\mathcal R)$.
Over a commutative ring $\mathcal R$, such generalized moment sequences are as expressive as LRS, i.e.\ every LRS can be expressed as a moment sequence and vice versa. For this reason, decidability results for generalized  moment sequences translate to decidability results for Skolem's problem and the positivity problem. 

In this paper, we study the decidability of the moment membership problem. 
That is, we consider the problem: For an $s \times s$ matrix $A$, decide whether 
$$ 
 \tr(A^n) \in \mathcal{P} \quad  \forall n\in \mathbb{N}
$$
where $\mathcal{P}$ is a fixed set. This set usually contains elements that are positive in some sense, so we call the problem also the \emph{moment positivity problem}. Most of our results also hold for generalized moments of the form $\varphi(A^n)$ as above.

One decisive factor in the algorithmic decidability of the problem is the instance set $\mathcal{D}$ of the matrices, which allows us to distinguish between our two main results: 
\begin{itemize}
	\item We restrict the instance set $\mathcal{D} \subseteq \mathrm{Mat}_s(\mathbb{Z})$ and prove decidability of the problem for a large subclass of integer matrices.
	\item We enlarge the instance set $\mathrm{Mat}_s(\mathbb{Z}) \subseteq \mathcal{D}$ and prove that the problem is undecidable 
for	matrices whose entries are elements of certain unital rings $\mathcal{R},$ for certain $\mathcal{P} \subseteq \mathcal{R}$.
\end{itemize}

\paragraph{Contributions.} Specifically, we determine the decidability of the moment membership problem in the following cases (see \cref{tab:contributions}):
\begin{itemize}
	\item \textbf{Decidability.} The moment positivity problem is decidable for orthogonal matrices (\cref{thm:moment_orthogonal}),  unitary matrices (\cref{cor:moment_unitary}), and matrices with a unique dominant eigenvalue or only real eigenvalues (\cref{thm:moment_decidable}).
	It follows that the positivity problem is decidable for simple unitary LRS, i.e.\ LRS whose characteristic polynomial only has simple  roots of modulus $1$,  
	as well as for LRS whose characteristic polynomial has a unique dominant root, or only real roots. 
	\item \textbf{Undecidability.} The generalized problem is undecidable for the ring of multivariate commutative polynomials (\cref{thm:undec-com}) as well as for non-commutative polynomials, where $\mathcal P$ is the set of polynomials with nonnegative coefficients (\cref{thm:moment_nonCommPoly}). 
	This implies that the corresponding positivity problem for LRS over commutative polynomials is undecidable.
	\item \textbf{Free P\'olya's Theorem.} As a side result, we prove a free version of P\'olya's theorem (\cref{thm:freePoyla}). We show that a non-commutative polynomial has nonnegative coefficients if and only if it is entrywise nonnegative on the set of entrywise nonnegative matrices.  
\end{itemize}

\renewcommand{\arraystretch}{1.2}

\begin{table}
\centering
\begin{tabular}{|c | c |}
\hline
\textbf{Decidable cases} & \textbf{Undecidable cases}\\
\hline
Unitary and Orthogonal matrices & Comm.\ polynomials $\mathbb{Z}[x_1, \ldots, x_d]$\\
(\cref{ssec:unitary}) & (\cref{ssec:commAlgebra}) \\
\hline
Dominant or real eigenvalue matrices & Non-comm.\ polynomials $\mathbb{Z}\langle z_1, \ldots, z_d \rangle$\\
(\cref{ssec:largestEigenvalue}) & (\cref{ssec:non-commAlgebra}) \\
\hline
\end{tabular}
\caption{For which instance sets is the (generalized) moment membership problem decidable or undecidable? This table summarizes the results of this paper.}
\label{tab:contributions}
\end{table}

\renewcommand{\arraystretch}{1}

This paper is structured as follows. In \cref{sec:problem} we present the problem statement and show the relation of moment problems to LRS. In \cref{sec:decidable} we present cases in which the moment problem is decidable. This includes a review of known results (\cref{ssec:known}), the decidability for orthogonal and unitary matrices (\cref{ssec:unitary}), and the decidability for matrices with unique largest eigenvalue or only real roots (\cref{ssec:largestEigenvalue}). In \cref{sec:undecidable}, we prove a version of the matrix mortality problem which implies that the moment problem is undecidable over certain commutative and non-commutative rings. Moreover, we prove a non-commutative version of P\'olya's Theorem.

\section{Problem statement}
\label{sec:problem}

Let  $\mathcal R$ be a unital ring (not necessarily commutative), and let
 $A\in \Mat_s(\mathcal R)$ be an $s\times s$ square matrix with entries from  $\mathcal R.$  For $n\geqslant 0$  the \emph{$n$-th moment of $A$} is defined as $$\mu_n(A) \coloneqq \tr\left(A^n\right)$$  where $\tr$ denotes the usual trace of a matrix, i.e.\ the sum of its diagonal entries.
The moments of $A$ are clearly elements from  $\mathcal R$, as for $A=\left(a_{ij}\right)_{i,j=1,\ldots, s}$ we have $$\mu_0(A)=\underbrace{1_{\mathcal R}+\cdots + 1_{\mathcal R}}_s \quad \mbox{ and }\quad  \mu_1(A)=\sum_{i=1}^s a_{ii},$$
and for $n\geqslant 2$
$$\mu_n(A)=\sum_{i_1,\ldots, i_n=1}^s a_{i_1i_2}\cdot a_{i_2i_3}\cdots \cdot a_{i_{n-1}i_n}\cdot a_{i_n i_1}.$$
Depending on the ring $\mathcal{R}$, the moments are studied in different contexts, as the following example shows.

\begin{example}
\label{ex:MPO_moment}
Let $V$ be a $\mathbb{C}$-vector space. Consider the tensor algebra $$\mathcal R \coloneqq T(V) \coloneqq \bigoplus_{m\geqslant 0} V^{\otimes m}= \mathbb{C} \oplus V \oplus \left( V\otimes V\right) \oplus \cdots $$
$\mathcal R$ forms a  unital ring with  tensor product as  multiplication (see \cite[Example 4.22]{Li21} for details).
Actually,  $\mathcal R$ is an $\mathbb N$-graded unital $\mathbb{C}$-algebra.

For  $A=\left(a_{ij}\right)_{i,j}\in \Mat_s(V)$ (which embeds into $\Mat_s(\mathcal R)$), we obtain
\begin{equation}\label{timpdo}
\mu_n(A)=\sum_{i_1,\ldots, i_n=1}^s a_{i_1i_2}\otimes a_{i_2i_3}\otimes  \cdots \otimes a_{i_{n-1}i_n}\otimes a_{i_n i_1},\end{equation} 
where the $n$-th moment is homogeneous of degree $n$ in $\mathcal R$. 
The expression in \eqref{timpdo} is called a \emph{translational invariant matrix product state} in tensor network theory. The number $s$ is usually called the  \emph{bond dimension} of the tensor $\mu_n(A)$.\end{example}

Now assume that $\mathcal R$ is also equipped with a subset $\mathcal P \subseteq \mathcal R$. In our results and applications, this  will always be a set of  elements that are {\it positive} in some sense. Further $\mathcal D\subseteq\mathrm{Mat}_s(\mathcal R)$ will be the set containing all instances of our decision problem. The general decision problem addressed in this paper is the following:

\begin{problem}[Moment positivity problem]\label{prob:moment_problem}
 Let $s$, $\mathcal{P}$, $\mathcal{D}$ be fixed as above.
For $A\in \mathcal D$ decide whether \emph{all} moments $\mu_n(A)$ belong to $\mathcal P$.
\end{problem}

Note that $\mathcal D,\mathcal P, s$ are fixed in our formulation of the problem. We are thus looking for an algorithm (tailored to $\mathcal R,\mathcal P, s$ and $\mathcal{D}$) that upon an input of any instance $A\in\mathcal D$ stops after a finite time, and returns yes if all moments of $A$ belong to $\mathcal P$, and no if at least one moment of $A$ does not belong to $\mathcal P$. If such an algorithm exist, we call the moment membership problem  \emph{decidable}, otherwise we call it \emph{undecidable}.

Note that if the ring operations are computable and membership of single elements in $\mathcal P$ is decidable, the moments membership problem is clearly semi-decidable in the following sense. Given $A\in \Mat_s(\mathcal R)$, we compute higher and higher moments of $A,$ and check membership in $\mathcal P$. If some moment does {\it not} belong to $\mathcal P$, we will know after a finite time. However, this algorithm runs forever in case that all moments do belong to $\mathcal P$. So the hard part of the problem is certifying membership of all moments in $\mathcal P$. We will make use of the semi-decidability in \cref{thm:moment_orthogonal}.

\subsection{Relation to the membership problem for linear recurrence sequences}

In the following, we review the relation of the moment problem to the positivity problem for linear recurrence sequences. 
A \emph{linear recurrence sequence} (LRS) $(u_n)_{n \in \mathbb{N}} \in \mathcal{R}^{\mathbb{N}}$ is a sequence whose elements are related to each other linearly, i.e.\
\begin{equation}\label{eq:LRS}
u_n = a_1 u_{n-1} + a_2 u_{n-2} + \cdots + a_s u_{n-s}
\end{equation}
for  all $n>s$. We call $s$ the \emph{order} of the recurrence relation. The positivity  problem for LRS is the following:

\begin{problem}\label{prob:lrs} 
Given an LRS as in \eqref{eq:LRS} with parameters $a_1, \ldots, a_s \in \mathcal{R}$ and  initial values $u_1, \ldots, u_s \in \mathcal{R}$, 
decide whether $u_n \in \mathcal{P}$ for all $n \in \mathbb{N}$.
\end{problem}

We start with the (well-known) observation that every generalized moment sequence is an LRS, if $\mathcal{R}$ is commutative.

\begin{lemma}\label{lem:mom_to_lin_rec} 
Let $\mathcal{R}$ be a commutative unital ring, and let $A \in \Mat_s(\mathcal{R})$. Then $\left(\varphi(A^n)\right)_{n\in\mathbb N}$ is an LRS of order $s$, for every $\mathcal{R}$-linear map $\varphi\colon \Mat_s(\mathcal R)\to\mathcal R$.\end{lemma}
\begin{proof}
Let $p(x) = x^s - a_1 x^{s-1} - \cdots - a_s$ be the characteristic polynomial of the matrix $A$. By the Cayley--Hamilton theorem for commutative rings (see for example \cite[Chapter XIV.3]{La02}), we have that
\begin{equation}\label{eq:CayleyHamilton}
A^s = a_1 A^{s-1} + a_2 A^{s-2} + \cdots + a_s I
\end{equation}
and therefore
$$A^n = a_1 A^{n-1} + a_2 A^{n-2} + \cdots + a_s A^{n-s}$$ for all $n\geqslant s.$ Applying $\varphi$ proves the statement.
\end{proof}

It is unclear whether a similar statement to \cref{lem:mom_to_lin_rec} is true for non-commutative rings. While there exist versions of the Cayley--Hamilton theorem for non-commutative rings (see for example \cite{Gr11, Sz06}), they cannot be applied to obtain an equation similar to \eqref{eq:CayleyHamilton}.

The next observation states that LRS are equivalent to generalized moment sequences. It can be found in  \cite{OW}:

\begin{lemma}\label{prop:lin_rec_equ_moment}
Let $(u_n)_{n \in \mathbb{N}}$ be a sequence in a commutative unital ring $\mathcal{R}$. The following are equivalent:
\begin{enumerate}
	\item[$(i)$] $(u_n)_{n \in \mathbb{N}}$ is a LRS of order $s$.
	\item[$(ii)$] There is a matrix $A \in \Mat_s(\mathcal{R})$ and two vectors $v,w \in \mathcal{R}^s$ such that $u_n = v^t A^{n-s} w$ for all  $n>s$.
\end{enumerate}
\end{lemma}
\begin{proof}
For $(i) \Rightarrow (ii)$ assume that the recurrence  is given by
$$u_n = a_1 u_{n-1} + a_2 u_{n-2} + \cdots + a_s u_{n-s}.$$ 
Using the companion matrix
$$A = \begin{pmatrix} a_1 & 1 & & &  \\ a_2 & & 1 & & \\  \vdots & & & \ddots & \\ a_{s-1} & & & & 1 \\ a_s & & & & \end{pmatrix} $$
we have that $u_{n} = v^t A^{n-s} w$ where
$ v = (u_s, u_{s-1}, \ldots, u_1)^t$  and $w = (1, 0, \ldots, 0)^t.$

The proof of $(ii) \Rightarrow (i)$ is analogous to \cref{lem:mom_to_lin_rec}, by replacing $\tr$ by the function $A \mapsto v^t A w$. Note that the recurrence starts to hold only for $n>2s$, but for our purposes this is irrelevant.\end{proof}

\section{Decidable cases}
\label{sec:decidable}

In the following we present cases in which the moment membership problem is decidable. This includes known results for small $s$ (\cref{ssec:known}), the moment positivity problem for unitary and orthogonal matrices (\cref{ssec:unitary}), and for matrices with a unique largest eigenvalue or only real eigenvalues(\cref{ssec:largestEigenvalue}). Throughout this section we will consider $\mathcal P=\mathbb R_{\geqslant 0}$.

\subsection{Known results: small order}
\label{ssec:known}

We first review known results on the decidability of the moment positivity problem. The results are about LRS, but in view of \cref{lem:mom_to_lin_rec}, they immediately transfer to moments. 

\begin{theorem}
\label{thm:moment_decidable_1}
The moment positivity problem is decidable in the following cases:
\begin{enumerate}[label=$(\roman*)$]
	\item $s\leqslant 5, \mathcal D=\mathrm{Mat}_s(\mathbb Q).$  
	\item $s\leqslant 9, \mathcal D\subseteq\mathrm{Mat}_s(\mathbb Q)$ the set of matrices with simple eigenvalues.
	\end{enumerate}
\end{theorem}
The proof of ($i$) is contained in \cite{OW}, the proof of ($ii$) goes back to \cite{OW2}. 
Decidability for other values of $s$ is unknown.

The positivity problem of LRS is closely related to \emph{Skolem's Problem} which asks if some sequence element equals $0$. The best result in this context is that Skolem's Problem is $\NP$-hard \cite{Bl01}. The decidability of the positivity problem implies decidability of Skolem's Problem. This follows for an integer LRS because $u_n \neq 0$  if and only if $u_n^2 - 1 \geqslant 0$. If $(u_n)_{n \in \mathbb{N}}$ is an LRS of order $s$, then $u_n^2 - 1$ is an LRS of order $s^2$. Moreover, since Skolem's Problem is $\NP$-hard, the positivity problem is $\NP$-hard as well.

\subsection{Orthogonal and unitary matrices}
\label{ssec:unitary}

We now show that the moment positivity problem for orthogonal (\cref{thm:moment_orthogonal}) and unitary matrices (\cref{cor:moment_unitary}) is decidable. The proof strategy is very similar to \cite{Bl03}.

A set $X \subseteq \mathbb{R}^m$ is algebraic if there are polynomials $p_1, \ldots, p_n\colon \mathbb{R}^m\to \mathbb{R}$ such that
$$X = \{x \in \mathbb{R}^m\mid p_1(x) = \cdots = p_n(x) = 0\}.$$
In this case, $X$ is the algebraic variety defined by $p_1, \ldots, p_n$, and we write $X = \mathcal{V}(p_1, \ldots, p_n)$. Even if the set of defining polynomials is infinite, there always exists a finite choice of polynomials defining the same algebraic variety, by Hilbert's basis theorem. Since we work over $\mathbb R$, we can even reduce it to a single polynomial by taking the sum of squares of the defining polynomials. 

For matrices $A_1, \ldots, A_d \in \Mat_{s}(\mathbb{R})$, let
$$\langle A_1, A_2, \ldots, A_d \rangle \coloneqq \big\{A_{k_1} \cdots A_{k_\ell}\mid \ell \in \mathbb{N}, \ k_1, \ldots, k_\ell = 1, \ldots, d\big\}$$
be the semigroup generated by $A_1, \ldots, A_d$. We denote by $\overline{\langle A_1, \ldots, A_d \rangle}$ the topological closure inside $\mathrm{Mat}_s(\mathbb R)$ with respect to the Euclidean topology.

\begin{lemma}\label{lem:compactAlgebraic}
Let $A_1, \ldots, A_d \in \mathrm{O}_s(\mathbb{Q})$ be orthogonal $s \times s$ matrices with rational entries. Then $\mathcal{G} \coloneqq \overline{\langle A_1, \ldots, A_d \rangle}$ is a compact algebraic group.
Moreover there is a recursively enumerable sequence of rational polynomials $(p_k)_{k \in \mathbb{N}}$ defining $\mathcal{G}$ inside $\mathrm{Mat}_s(\mathbb R)$.

\end{lemma}
\begin{proof}
Compactness of $\mathcal{G}$ is obvious. To prove that $\mathcal{G}$ is a group we only have to show that $A^{-1} \in \mathcal{G}$ for every $A \in \mathcal{G}$. Consider the sequence $(A^k)_{k \in \mathbb{N}}$. By compactness, there exists a converging subsequence. In other words, for every $\varepsilon > 0$, there exists $n_2 > n_1 +1$ such that
$$\Vert A^{n_1} - A^{n_2} \Vert < \varepsilon$$
where $\Vert \cdot \Vert$ is the operator norm. Since $\Vert A \cdot B \Vert = \Vert B \Vert$ for every matrix $B$, we obtain
$$\Vert A^{-1} -  A^{n_2 - n_1 - 1} \Vert < \varepsilon.$$
This shows that $A^{-1} \in \mathcal{G}$.

Now note that every compact group $\mathcal{G} \subseteq \Mat_s(\mathbb{R})$ is algebraic (see for example \cite[Chapter 3, Section 4.4]{On90}). In particular, it is shown there that
$$\mathcal{G} = \mathcal{V}\left(\mathbb{R}[X]^\mathcal{G} \right) \coloneqq \mathcal{V}\Big(p \in \mathbb{R}[X]\mid  p(I_s) = 0, \ p(gX) = p(X) \text{ for all } g \in \mathcal{G} \Big),$$
where $I_s$ is the identity matrix of size $s$.

Now  note that if $\mathcal{G}$ is generated by $A_1, \ldots, A_d$, then the invariance only needs to be checked with respect to the generators, i.e.\
$$\mathcal{G} = \mathcal{V}\Big(p \in \mathbb{R}[X]\mid p(I_s) = 0, \ p(A_i X) = p(X) \ \text{ for } i=1, \ldots,d \Big).$$
Since the conditions $p(I_s) = 0$ and $p(A_i X) = p(X)$ are linear in the coefficients of $p$, there exists a basis $(p_k)_{k \in \mathbb{N}}$ of the space of solutions of these conditions. Moreover, the coefficients of the basis vectors $p_k$ can be chosen from $\mathbb{Q}$, since all conditions are rational. We now have
$$\mathcal{G} = \mathcal{V}(p_k\mid  k \in \mathbb{N}).$$
The polynomials $p_k$ can be computed recursively by solving the system of linear equations over the space of polynomials with degree $d,$ and by increasing $d$ iteratively. 
\end{proof}

Note that the statement is not true if $\mathbb{R}$ is replaced by $\mathbb{C}$. For example the group
$$\mathcal{G} \coloneqq \left\{e^{i \theta}\mid \theta \in [0,2\pi)\right\},$$
seen as a subset of $1 \times 1$ matrices, is not algebraic. Yet, we shall see that the moment problem also generalizes to unitary matrices (see \cref{cor:moment_unitary}).

Since $\mathbb{R}[X]$ is a Noetherian ring, there exists $n \in \mathbb{N}$ such that 
$$ \mathcal{G} = \mathcal{V}(p_1,\ldots, p_n).$$
This will be an important ingredient to prove the decidability of the moment problem. Note however that $n$ can be arbitrarily large and it is unclear whether $n$ is computable or not.

\begin{theorem}
\label{thm:moment_orthogonal}
The moment positivity problem for  $\mathcal D=\mathrm{O}_s(\mathbb{Q})$  is decidable.
\end{theorem}
\begin{proof}
We will present two procedures, each certifying either yes- or no-instances in finite time. Letting these algorithms run in parallel will result in a decision algorithm for the problem.

Certifying no-instances for $A \in \mathrm{O}_s(\mathbb{Q})$ is achieved by iteratively checking whether $\tr(A^n) \geqslant 0$ holds for every  $n$. If $A$ is a no-instance, this algorithm will halt when detecting $\tr(A^n) < 0$ for the first time.

We now present an algorithm to certify yes-instances in finite time. For a given $A \in \mathrm{O}_s(\mathbb{Q})$, the moment membership problem can be rephrased  as
$$\forall B \in \langle A \rangle: \tr(B) \geqslant 0.$$
By the continuity of the trace, this is equivalent to 
\begin{equation}
\label{eq:firstOrderStatement}
\forall B \in \overline{\langle A \rangle}: \tr(B) \geqslant 0.
\end{equation}
By \cref{lem:compactAlgebraic} there exists a recursively enumerable sequence of polynomials $(p_k)_{k \in \mathbb{N}}$ and some $n \in \mathbb{N}$ such that
$$\overline{\langle A\rangle} = \mathcal{V}(p_1,\ldots, p_n).$$
Now step $k$ of the algorithm verifies the statement
\begin{equation}\label{eq:VarietyStatement}
\forall B \in \mathcal{V}(p_1, \ldots, p_k)\colon \tr(B) \geqslant 0
\end{equation}
which is decidable by the Tarski--Seidenberg Theorem, since it is a statement in first order logic. As soon  as  \cref{eq:VarietyStatement} is true for the first time, the algorithm halts and outputs a correct yes-answer. This  will be the case after at most $n$ steps, if $A$ is a yes-instance.
\end{proof}

\begin{remark}\label{rem:generalizations} The previous statement can be generalized in two directions:
\begin{enumerate}[label=($\roman*$)] 
	\item By the same argument, the following problem is also decidable: 
		Given $A_1, \ldots A_d\in \mathrm{O}_s(\mathbb{Q})$ for a fixed matrix size $s$, decide if:
	$$\forall \ell \in \mathbb{N}\  \forall k_1, \ldots, k_\ell \in \{1, \ldots, d\}\colon \tr(A_{k_1} \cdots A_{k_\ell}) \geqslant 0.$$
	Note that this decision problem for arbitrary matrices is undecidable \cite{DCCW}.
	\item\label{generalizations:ii} The proof remains true if $\tr$ is replaced by any other polynomial function. This in particular implies that the generalized problem
	$$\forall n \in \mathbb{N}\colon \varphi(A^n) \geqslant 0$$
	is decidable.
\end{enumerate}
\end{remark}

We now generalize the result to unitary matrices, by embedding them into orthogonal matrices of larger size. We denote by $\mathbb{Q}[i]$ the field of complex numbers with rational real and imaginary parts, and we denote the set of $s \times s$ unitary matrices with entries in $\mathbb{Q}[i]$ by $\mathrm{U}_s(\mathbb{Q}[i])$.

\begin{lemma}
\label{lem:UtoO_embedding}
The map
\begin{align*}\Psi \colon \mathrm{U}_s(\mathbb{Q}[i]) &\to \mathrm{O}_{2s}(\mathbb{Q}) \\ U = A + iB &\mapsto \begin{pmatrix}
A & -B \\ B & A
\end{pmatrix}
\end{align*}
is a group homomorphism. Moreover we have 
$$\tr(U) = \frac{1}{2} \tr\left(\Psi(U) \cdot \begin{pmatrix} I_s & i I_s \\ -i I_s & I_s \end{pmatrix} \right).$$
\end{lemma}
\begin{proof}
The map is well defined since $\Psi(U)$ is orthogonal if and only if $U$ is unitary. The rest is immediate.
\end{proof}

The main results of this section are summarized in the following two corollaries.

\begin{corollary}
\label{cor:moment_unitary}
For each $s\geqslant 1$, the moment positivity problem for matrices from $\mathrm{U}_s(\mathbb Q[i])$ is decidable.
\end{corollary}
\begin{proof}
It follows immediately from \cref{lem:UtoO_embedding}, \cref{thm:moment_orthogonal} and \cref{rem:generalizations} ($ii$).
\end{proof}

\begin{corollary}
The positivity problem is decidable for simple unitary LRS, i.e.\
$$
u_n = a_1 u_{n-1} + \cdots + a_s u_{n-s}
$$
with $a_1, \ldots, a_s \in \mathbb{Q}[i]$, where the roots of $p(x) = x^s - a_1 x^{s-1} - a_2 x^{s-2} - \cdots - a_{s}$
are all simple and of modulus $1$.
\end{corollary}
\begin{proof}
We choose a unitary matrix $A\in{\rm U}_s(\mathbb C)$ whose eigenvalues are the roots of $p$, and whose entries are computable numbers. For example, one can take a diagonal matrix with the specified roots on the diagonal. We obtain the reccurence $$A^n=a_1A^{n-1}+\cdots + a_s A^{n-s}$$ for all $n\geqslant s$, and since the roots are all simple, $p$ is actually the minimal polynomial of $A$. So $I_s,A,A^2,\ldots, A^{s-1}$ are linearly independent, and we can thus find a linear functional $\varphi$ on $\Mat_s(\mathbb C)$ with $\varphi(A^i)=u_i$ for $i=0,\ldots, s-1$. Now, by \cref{rem:generalizations} and \cref{lem:UtoO_embedding} above, it is decidable whether $\varphi(A^i)\geqslant 0$ holds for all $i$, and since this sequence fulfills the same  recurrence and initial conditions as $(u_i)_{i\geqslant 1}$, the two sequences coincide.\end{proof}

While the positivity problem is known to be $\NP$-hard for linear recurrence sequences \cite{Bl01}, the complexity of this problem for unitary LRS remains open.

\subsection{Matrices with a unique dominant eigenvalue or real eigenvalues}
\label{ssec:largestEigenvalue}

In the following, we show that for matrices with a unique  dominant eigenvalue, and  for matrices with only real eigenvalues, the moment problem is decidable. Note that the idea for the case of a unique dominating eigenvalue is already present in \cite{OW}, but restricted to  multiplicity $1$ and matrices of size at most $s=5$.
\begin{theorem} 
\label{thm:moment_decidable}
The moment positivity problem is decidable in the following cases:

\begin{enumerate}[label=$(\roman*)$]
	\item\label{cond:moment_dec1} $\mathcal R=\mathbb Q, s$ arbitrary, and the set of instances restricted to matrices with a unique dominant eigenvalue. 
	\item\label{cond:moment_dec3} $\mathcal R=\mathbb Q, s$ arbitrary, and the set of instances restricted to matrices with only real eigenvalues.
\end{enumerate}
\end{theorem}

\begin{proof} We provide algorithms that decide the moments positivity problem for the stated instance sets. We can assume without loss of generality that $A \in \Mat_s(\mathbb{Z}),$ by possibly multiplying the matrix with the largest denominator of its entries.

For \ref{cond:moment_dec1} let $A \in \Mat_s(\mathbb Z)$ have a unique dominant eigenvalue. Since $A$ has real entries, the non-real eigenvalues of $A$ come in conjugate pairs. Since there is exactly one eigenvalue $\lambda_1$ of largest absolute value, it must be real. We  let $k$ denote its multiplicity and obtain $$|\mu_n(A) -  k\cdot \lambda_1^n| \leqslant (s-k) |\lambda_2|^n,$$
where $\lambda_2$ denotes the second largest eigenvalue in absolute value. 
Thus it suffices to check $\mu_n(A) \geqslant 0$ for $n$ up to 
$$\frac{\log(s/k-1)}{ \log(|\lambda_1|)-\log(|\lambda_2|)}.$$

\ref{cond:moment_dec3}: In this case only odd moments matter, since the even moments are always nonnegative. If the dominant eigenvalues all have the same sign, then we can apply \ref{cond:moment_dec1}. Otherwise,  since odd powers of eigenvalues with the same absolute values but different signs cancel out, we can reduce the problem to a smaller matrix, where the dominant eigenvalues do have the same sign.
\end{proof}

\subsection{Further generalizations}

In the following, we present a generalization of the statements in \cref{ssec:unitary} and \cref{ssec:largestEigenvalue}.
For a matrix $A\in \Mat_s(\mathbb R)$, we denote by $\mathsf{spec}(A)$ the multi-set of all eigenvalues of $A$ (where multiple eigenvalues are represented by multiple elements of $\mathsf{spec}(A)$). Express
$$\mathsf{spec}(A) = \mathsf{per}_1(A) \cup \mathsf{per}_2(A) \cup \cdots \cup \mathsf{per}_s(A)$$
as a partition of peripheral spectra, i.e.\ eigenvalues of the same absolute value, in decreasing order (i.e.\ $\mathsf{per}_1(A)$ contains the dominant eigenvalues, $\mathsf{per}_2(A)$ the eigenvalues of second largest absolute value...). Note that $\mathsf{per}_{i}(A)$ can be empty if $A$ has multiple eigenvalues of same absolute value.
Moreover, let
$$ \mu_n^{(i)}(A) \coloneqq \sum_{\lambda \in \mathsf{per}_i(A)} \left(\frac{\lambda}{|\lambda|}\right)^n.$$ 
We define
$$\eta_i(A) = \left\{\begin{array}{cl} \inf_{n \in \mathbb{N}} \mu^{(i)}_n(A) & \text{ : if } \mathsf{per}_i(A) \neq \emptyset \\ \infty & \text{ : if } \mathsf{per}_i(A) = \emptyset \end{array}\right.$$
and
$$\gamma_i(A) =  \left\{\begin{array}{cl} \sup_{n \in \mathbb{N}} \mu^{(i)}_{pn+q}(A) & \text{ : if } \mathsf{per}_i(A) \neq \emptyset \\ -\infty  & \text{ : if } \mathsf{per}_i(A) = \emptyset\end{array}\right.,$$ where $p,q\geqslant 1$ are arbitrary but fixed integers. So we compute the supremum  along an arithmetic progression.

\begin{lemma}\label{lem:mu_eta_decidability}
The following two problems are decidable:
\begin{enumerate}[label=$(\roman*)$]
	\item\label{prob1:lemma} Given\footnote{To assume that the inputs have a finite description, we restrict to algebraic numbers, i.e.\ numbers that can be represented as roots of an integer polynomial. This is enough to apply \cref{lem:mu_eta_decidability} in the proof of \cref{thm:generalVersion}.} $A \in \mathrm{Mat}_s(\mathbb{R})$, $c \in \mathbb{R},$ decide whether $\eta_i(A) \geqslant c.$
	\item\label{prob2:lemma} Given $A \in \mathrm{Mat}_s(\mathbb{R})$, $c \in \mathbb{R},$ decide whether $\gamma_i(A) \leqslant c.$
\end{enumerate}
\end{lemma}
\begin{proof}
The decision algorithms are very similar to one from the proof of \cref{thm:moment_orthogonal}. To construct an algorithm for ($i$), let the following two procedures run in parallel:
\begin{enumerate}[label=($\alph*$)]
	\item\label{alg1a} Evaluate $\mu_n^{(i)}(A)$ for increasing $n \in \mathbb{N}$. Halt if $\mu_n^{(i)}(A) < c$.
	\item\label{alg2a} Check the statement
	$$ \forall B \in \mathcal{V}(p_1, \ldots, p_k)\colon \tr(B) \geqslant c$$
	for increasing $k \in \mathbb{N}$, where $(p_\ell)_{\ell \in \mathbb{N}}$ define the variety $\overline{\langle U \rangle},$ where $U$ is the diagonal matrix with eigenvalues $\lambda/|\lambda|$ for $\lambda \in \mathsf{per}_i(A)$. Halt if the statement is true.
\end{enumerate}
If $A,c$ is a no-instance of ($i$), then ($a$) will eventually halt; if $A,c$ is a yes-instance, ($b$) will eventually halt, for the same reason as in the proof of \cref{thm:moment_orthogonal}.

The algorithm for ($ii$)  is very similar. Let the following two procedures run in parallel:
\begin{enumerate}[label=($\alph*$)]
	\item\label{alg1b} Evaluate $\mu_{pn+q}^{(i)}(A)$ for increasing $n \in \mathbb{N}$. Halt if $\mu_{pn+q}^{(i)}(A) > c$.
	\item\label{alg2b} Check  the statement
	$$ \forall B \in \mathcal{V}(p_1, \ldots, p_k)\colon \tr(U^qB) \leqslant c$$
	for increasing $k \in \mathbb{N}$, where $(p_\ell)_{\ell \in \mathbb{N}}$ define the group $\overline{\langle U^p \rangle},$ where $U$ is the diagonal matrix with eigenvalues $\lambda/|\lambda|$ for $\lambda \in \mathsf{per}_i(A)$. Halt if the statement is true.
\end{enumerate}
In ($b$) we only evaluate odd moments; recall  \cref{rem:generalizations} ($ii$). 
\end{proof}

It is unclear whether $\eta_i(A) > c$ or $\eta_i(A) = c$ is decidable. This is due to the fact that we do not know whether $\mu_n^{(i)}(A)$ attains the infimum/supremum for finite $n$.

\begin{theorem}\label{thm:generalVersion}
For a fixed parameter $\varepsilon > 0$, the moment positivity problem is decidable for all non-zero matrices $A$ satisfying one of the following conditions:
\begin{enumerate}[label=$(\roman*)$]
	\item\label{cond:generalVers1} $\exists k \in \mathbb{N}\colon \eta_1(A), \ldots, \eta_k(A) \geqslant 0, \eta_{k+1}(A) \geqslant \varepsilon.$
	\item\label{cond:generalVers2} $\exists k \in \mathbb{N}\colon \gamma_1(A), \ldots, \gamma_k(A) \leqslant 0, \gamma_{k+1}(A) \leqslant -\varepsilon.$
	\item\label{cond:generalVers3} $\eta_1(A) < 0.$
\end{enumerate}
If $(ii)$ or $(iii)$ are satisfied, then $A$ is automatically a no-instance. If $(i)$ is satisfied, then $A$ can be a yes or a no-instance. Moreover, each of the above criteria is decidable.
\end{theorem}
\begin{proof}
First, checking whether $A$ satisfies ($i$), ($ii$) or ($iii$) is decidable by \cref{lem:mu_eta_decidability}, since there are only finitely many of these statements to check.

To prove ($i$), assume that $\eta_{k+1}(A) \neq \infty$ (the other case is trivial). Let $\lambda_i \in \mathsf{per}_i(A)$. We have that
$$\mu_n(A) = \sum_{i=1}^d |\lambda_i|^n \mu_n^{(i)}(A) \geqslant |\lambda_{k+1}|^n \left( \varepsilon - s \sum_{i=k+2}^d \left(\frac{|\lambda_{i}|}{|\lambda_{k+1}|}\right)^n \right)$$
which is positive for $$n\geqslant \frac{\log(\varepsilon)-\log(sd)}{\log(\vert\lambda_{k+2}\vert)-\log(\vert\lambda_{k+1}\vert)}.$$ So we only need to check finitely many instances of the problem.

For ($ii$) we have that
$$\mu_m(A) = \sum_{i=1}^d |\lambda_i|^m \mu_m^{(i)}(A) < |\lambda_{k+1}|^m \left( - \varepsilon + s \sum_{i=k+2}^d \left(\frac{|\lambda_{i}|}{|\lambda_{k+1}|}\right)^m \right).$$
Now there clearly exists some $m$ of the form $pn+q$ such that the right hand side is negative.

For ($iii$) note that $\eta_1(A) < 0$ is decidable since $\eta_1(A) \geqslant 0$ is decidable by \cref{lem:mu_eta_decidability}. 
Let $0 < \delta < -\eta_1(A)$. Then there exists an increasing sequence $(n_{\ell})_{\ell \in \mathbb{N}}$ such that $\mu_{n_{\ell}}^{(1)}(A) < \eta_1(A) + \delta < 0$ for all $\ell$.\footnote{This follows from the fact that for a unitary matrix $U$ the group $\overline{\{U^n\mid n \in \mathbb{N}\}}$ is either finite or contains no isolated points. Because if the set contains an isolated point, then all elements are isolated. But a compact set which contains only isolated points is finite. Hence there exists an increasing sequence $(n_\ell)_\ell$ such that $$\eta_i(U) \leqslant \tr(U^{n_\ell}) \leqslant \eta_i(U) + \frac{1}{\ell}.$$}
Therefore we have
$$\mu_{n_\ell}(A) = \sum_{i=1}^d |\lambda_i|^{n_\ell} \mu_{n_\ell}^{(i)}(A) < |\lambda_1|^{n_{\ell}} \left(\eta_1(A) + \delta +  s \sum_{i=2}^d \left(\frac{|\lambda_{i}|}{|\lambda_{1}|}\right)^{n_\ell}\right).$$
Again there exists $\ell_0$ such that $\mu_{n_{\ell_0}}(A) < 0$.
\end{proof}

\section{Undecidable cases}
\label{sec:undecidable}

We now present two finitely generated rings for which the moment membership problem is undecidable. Specifically, in \cref{ssec:commAlgebra} we prove that the moment membership problem is undecidable for the ring of commutative polynomials $\mathcal{R} = \mathbb{Z}[x_1, \ldots, x_d]$ if $n$ is sufficiently large. In \cref{ssec:non-commAlgebra}, we show that the moment membership problem is also undecidable for the space of non-commutative polynomials $\mathcal{R} = \mathbb{Z}\langle z_1, \ldots, z_d \rangle$. For a similar undecidable problem including nonnegative and sos polynomials over infinitely many variables, we refer to \cite[Section 6]{De21b}.

\subsection{Commutative polynomial rings}
\label{ssec:commAlgebra}

In the following, we show that the generalized moment membership problem for  $\mathcal R=\mathbb{Z}[x_1, \ldots, x_d]$ and the cone
$$\mathcal{P}_{\rm coeff} \coloneqq \{p \in \mathbb{Z}[x_1, \ldots, x_d]\mid  \text{all coefficients of } p \text{ are nonnegative}\}$$
is undecidable. The proof relies on the following undecidable problem.

\begin{lemma}\label{lem:trace_undec} For large enough values of $s$ and $d$, and a suitable matrix $N \in \mathrm{Mat}_s(\mathbb{Z}),$  the following 
problem is undecidable: Given $A_1, \ldots, A_d \in \mathrm{Mat}_s(\mathbb{Z})$, do there exist  $ n_1, \ldots, n_d \in \mathbb{N}$ with $$\tr(A_1^{n_1} \cdot A_2^{n_2} \cdots A_d^{n_d} \cdot N) < 0?$$
\end{lemma}

In order to prove \cref{lem:trace_undec}, we use a similar proof idea as in \cite{DCCW, Kl22}, which is based on a reduction from an undecidable version of the matrix mortality problem. We start with a version of the matrix mortality problem known to be undecidable.

\begin{problem}[Matrix Mortality version]\label{prob:mortality_version}
Let $A_1, \ldots, A_d \in \mathrm{Mat}_s(\mathbb{Z})$. Does there exist a choice of $n_1, \ldots, n_d \in \mathbb{N}$ such that
$$ A_1^{n_1} \cdot A_2^{n_2} \cdots A_{d}^{n_d} = 0?$$
\end{problem}

\begin{proposition}\label{thm:mortality_version}
If $s$ and $d$ are large enough, then \cref{prob:mortality_version} is undecidable.
\end{proposition}
For a proof of \cref{thm:mortality_version} we refer to \cite{Be08}.
We are now ready to prove \cref{lem:trace_undec}.

\begin{proof}[Proof of \cref{lem:trace_undec}]
We provide a reduction from \cref{prob:mortality_version}.
First, fix the matrix 
$$ N = \begin{pmatrix} \mathbf{0} & 0 \\ 0 & 1 \end{pmatrix} + \sum_{i,j = 1}^{s} \begin{pmatrix} E_{ij} \otimes E_{ij} & 0 \\ 0 & 0\end{pmatrix} \in \Mat_{s}(\mathbb{Z})^{\otimes 2} \oplus \mathbb{Z} \subseteq  \Mat_{s^2 + 1}(\mathbb{Z})$$
where $E_{ij} = e_{i} e_{j}^t$ with $e_k$ being the $k$th standard vector. For every matrix in $\Mat_{s^2 + 1}(\mathbb{Z})$ of the form
$$Y = \begin{pmatrix} X \otimes X & 0 \\ 0 & a \end{pmatrix}$$
we have 
$$ \tr(YN) = a + \sum_{i,j = 1}^{s} X_{ij}^2.$$
For an instance $A_1, \ldots, A_d \in \mathrm{Mat}_s(\mathbb{Z})$ of \cref{prob:mortality_version}, define the following $d+1$ matrices:
$$B_i = \begin{pmatrix}
A_i \otimes A_i & 0 \\ 0 & 1
\end{pmatrix} \quad \text{ for } i = 1, \ldots, d$$
and
$$B_{d+1} = \begin{pmatrix} I_{s}\otimes I_s & 0 \\ 0 & -1 \end{pmatrix}$$
where $I_s$ is the identity matrix of size $s$.

Let $n_1, \ldots, n_d \in \mathbb{N}$ be such that 
$$A_1^{n_1} \cdots A_d^{n_d} = 0.$$
Choosing $n_{d+1}=1$ we obtain
$$\tr\left(B_1^{n_1} \cdots B_{d+1}^{n_{d+1}} \cdot N\right) = -1 + \sum_{i,j = 1}^{s} \left(A_1^{n_1} \cdots A_d^{n_d}\right)_{ij}^2 = -1 < 0.$$
Conversely, let $n_1, \ldots, n_{d+1} \in \mathbb{N}$ be such that $\tr(B_1^{n_1} \cdots B_{d+1}^{n_{d+1}} \cdot N)  < 0$. This is only possible for $n_{d+1}$ odd  and 
$$\sum_{i,j = 1}^{s} \left(A_1^{n_1} \cdots A_d^{n_d}\right)_{ij}^2 = 0,$$
which implies  $A_1^{n_1} \cdots A_d^{n_d} = 0$.
\end{proof}

We now state the main problem and theorem of this section, namely the membership problem of the moment problem for matrices over non-commutative rings.

\begin{problem} 
\label{prob:moment_alt}
Let $M \in \mathrm{Mat}_s(\mathcal{R})$ be a fixed matrix.
For an input $A \in \mathrm{Mat}_s(\mathcal{R})$, decide whether
$$ \tr(A^n \cdot M) \in \mathcal{P}$$ holds for all $n\geqslant 1$.
\end{problem}

\begin{theorem}\label{thm:undec-com}
If $s, d \in \mathbb{N}$ are large enough, and $M$ is chosen suitably, then \cref{prob:moment_alt} is undecidable for $\mathcal{R} = \mathbb{Z}[x_1, \ldots, x_d]$ and $\mathcal{P}_{\rm coeff}$.
\end{theorem}

\begin{proof}
Given $A_1,\ldots, A_d\in \Mat_{s}(\mathbb Z)$, set
$$A =\sum_{i=1}^d \left(\sum_{1 \leqslant j \leqslant i} e_j e_i^t  \right) \otimes A_i \cdot x_i \in \Mat_{ds}(\mathcal R).$$
Moreover, define
$$ M = m m^t \otimes N $$
with $m = (1, \ldots, 1)$ and $N$ as in \cref{lem:trace_undec}. We have that
\begin{align*}
\tr(A^n M) &= \sum_{1 \leqslant i_1 \leqslant \cdots \leqslant i_n \leqslant d} i_1\cdot  \tr(A_{i_1} \cdots A_{i_n} \cdot N) \cdot x_{i_1} \cdots x_{i_n}
\\ &= \sum_{n_1 + \cdots + n_d = n}\underbrace{c_{n_1,\ldots, n_d}}_{\geqslant 1} \cdot \tr(A_1^{n_1} \cdots A_{d}^{n_d} \cdot N)\cdot x_1^{n_1} \cdots x_d^{n_d}
\end{align*} where $c_{n_1,\ldots, n_d}=\min\{i\mid n_i\neq 0\}.$
Thus \cref{prob:moment_alt} reduces to the undecidable problem from \cref{lem:trace_undec}.
\end{proof}

\begin{remark} Since the sequence $\tr(A^nM)$ is clearly a LRS (see \cref{lem:mom_to_lin_rec}), the previous result shows that positivity of LRS over $\mathcal R=\mathbb Z[x_1,\ldots,x_d]$ is undecidable in general.
\end{remark}

\subsection{Non-commutative polynomial rings}
\label{ssec:non-commAlgebra}

We now consider  the ring $\mathcal R=\mathbb Z\langle z_1,\ldots, z_d\rangle$ of non-commutative polynomials, and show that its moment membership problem is undecidable for the cone of polynomials with positive coefficients.
As a $\mathbb Z$-module,  a basis of $\mathcal R$ consists of all words in the letters $z_1,\ldots, z_d$ where the order of letters {\it does} matter. Concatenation of words extends to a multiplication making $\mathcal R$ a unital ring, where $1$ corresponds to the empty word. 
There is a slightly different way to define this object, namely just as the tensor algebra (cf.\ \cref{ex:MPO_moment})
$$\mathbb Z\langle z_1,\ldots, z_d\rangle= T(\mathbb Z^d).$$ 
The equivalence of definitions is apparent when identifying a word $z_{k_1}\cdots z_{k_m}$ with the element $e_{k_1}\otimes \cdots \otimes e_{k_m}\in \left(\mathbb Z^d\right)^{\otimes m},$ where $e_r$ denotes the $r$-th standard basis vector in $\mathbb Z^d$.

We equip $\mathcal R$ with two (a priori) different sets of positive elements: 
$$\begin{aligned}\mathcal P_{\rm coeff} &\coloneqq \mathbb Z_{\geqslant  0}\langle z_1,\ldots, z_d\rangle=\left\{ p\in \mathbb Z\langle z_1,\ldots, z_d\rangle\mid\mbox{all coefficients of } p \mbox{ are nonnegative}\right\}\\
\mathcal P_{\rm eval} &\coloneqq \left\{ p\in \mathbb Z\langle z_1,\ldots, z_d\rangle\mid\forall \ell, A_1,\ldots, A_d\in \Mat_\ell(\mathbb Z_{\geqslant 0})\colon p(A_1,\ldots, A_d)\in \Mat_\ell(\mathbb Z_{\geqslant 0})\right\}. \end{aligned}$$
We now show that both cones coincide, which is a free version of P\'olya's Theorem.

\begin{theorem}[Free P\'olya's Theorem]\label{thm:freePoyla}
Let $p\in\mathbb C\langle z_1,\ldots, z_d\rangle$ with $m \coloneqq \deg(p)$. Then the following are equivalent:
\begin{enumerate}[label=(\roman*)]
\item\label{cond:Polya1} All coefficients of $p$ are nonnegative reals.
\item\label{cond:Polya2} For all  $A_1,\ldots, A_d\in \Mat_{m+1}(\mathbb Z_{\geqslant 0})$ we have
$$p(A_1,\ldots, A_d)\in \Mat_{m+1}(\mathbb R_{\geqslant 0}).$$
\end{enumerate} In particular, $\mathcal P_{\rm coeff}=\mathcal P_{\rm eval},$ and in the definition of $\mathcal P_{\rm eval}$ one can restrict $\ell$ to $\deg(p)+1.$
\end{theorem}
\begin{proof} ($i$) $\Rightarrow$ ($ii$) is obvious (even without the restriction on the matrix size $m+1$). 
For ($ii$) $\Rightarrow$ ($i$) we  construct matrices $A_1,\ldots, A_d$ that allow us to isolate a single coefficient of $p$. 

Let  $z_{k_1}\cdots z_{k_\ell}$ be a word in the letters $z_1,\ldots, z_d$. For $j=1,\ldots, d$ define
$$ A_j \coloneqq \sum_{i=1,\ldots, \ell;\ k_i=j} E_{i,i+1}\in \Mat_{\ell+1}(\mathbb Z_{\geqslant 0}),$$ where $E_{i,j}$ denotes the matrix (of size $\ell+1$) with a $1$ in position $(i,j)$ and zeros elsewhere. For $t_1,\ldots, t_r\in\{1,\ldots, d\}$ we have $$A_{t_1}\cdots A_{t_r}=\sum_{\tiny\begin{array}{c}i \\ k_i=t_1\\k_{i+1}=t_2\\ \vdots \\ k_{i+r-1}=t_r\end{array} }E_{i,i+r}\in \Mat_{\ell+1}(\mathbb Z_{\geqslant 0}).$$ In particular, the $(1,\ell+1)$-entry of a product  $A_{t_1}\cdots A_{t_r}$ is $1$ if and only if 
$r=\ell$ and $(k_1,\ldots, k_\ell)=(t_1,\ldots, t_\ell)$; in all other cases it is zero. So $p(A_1,\ldots, A_d)$ contains in its upper right entry precisely the coefficient of $p$ at the word $z_{k_1}\cdots z_{k_\ell}.$

Since all words appearing in $p$ are of length at most $\deg(p)=m$, we can do this procedure with matrices $A_j$ of size at most $m+1$, and thus clearly with matrices of size exactly $m+1$.
\end{proof}

\begin{remark}
P\'olya's theorem \cite{Po28,Ha52} states that for every homogeneous polynomial $p \in \mathbb{R}[x_1, \ldots, x_d]$ that is strictly positive  on the $d$-simplex
$$\Delta_d \coloneqq \left\{(a_1, \ldots, a_d) \in \mathbb{R}^d \biggm| a_i \geqslant 0, \sum_{i = 1}^d a_i = 1 \right\},$$
the polynomial 
$$ (x_1 + \cdots + x_d)^n \cdot p(x_1, \ldots, x_d)$$
has positive coefficients, for sufficiently large $n \in \mathbb{N}$.
In \cref{thm:freePoyla}, the space of nonnegative matrices takes the role of the $d$-simplex. While in the commutative case we have to multiply $p$ with an additional polynomial, this is not so in the free version. 

To the best of our knowledge, this is the first result that shares these key features with Pólya's theorem, specifically linking the positivity of the coefficients of a polynomial with positivity on the nonnegative orthant.

\end{remark}

We now show that for these cones the moment membership problem is undecidable.

\begin{theorem} 
\label{thm:moment_nonCommPoly} Let  $d,s\geqslant 7$. Then  the moment membership problem for $\mathcal R=\mathbb Z\langle z_1,\ldots, z_d\rangle$, $\mathcal P_{\rm coeff}=\mathcal P_{\rm eval}$ and $s$ is undecidable. This remains true if we restrict the instances to linear matrix polynomials, i.e.\  
$A\in \Mat_s(\mathbb Z\langle z_1,\ldots, z_d\rangle)$ whose entries are linear forms in $z_1,\ldots, z_d$.
\end{theorem}

\begin{proof}
For $A=\sum_{k=1}^d z_kA_k$ with $A_k\in \Mat_s(\mathbb Z)$ we have $$\mu_n(A)=\sum_{k_1,\ldots, k_n=1}^d \tr(A_{k_1}\cdots A_{k_n})\cdot z_{k_1}\cdots z_{k_n}.$$ So $\mu_n(A)\in\mathcal P_{\rm coeff}$ means that $\tr(A_{k_1}\cdots A_{k_n})\geqslant 0$ for all $k_1,\ldots, k_n=1,\ldots, d$. Undecidability of this problem was proven in \cite[Lemma 3]{DCCW}. 
\end{proof}

\section{Conclusion}
\label{sec:conclusion}

We have studied the moment membership problem (\cref{prob:moment_problem}) for matrices over a ring. We have shown that there is a relation to LRS for commutative rings (\cref{lem:mom_to_lin_rec} and \cref{prop:lin_rec_equ_moment}) and that the moments positivity problem is decidable in many cases, including unitary and orthogonal matrices (\cref{thm:moment_orthogonal} and \cref{cor:moment_unitary}) as well as matrices with a unique dominating eigenvalue or only real eigenvalues (\cref{thm:moment_decidable}). Finally, we have shown that the generalized moment membership problem is undecidable over the ring of commutative and non-commutative polynomials, where the positivity cone is given by the set of polynomials with non-negative coefficients (\cref{thm:undec-com} and \cref{thm:moment_nonCommPoly}).

The central open question remains, namely whether the moment membership problem is decidable or undecidable for $\mathcal{R} = \mathbb{Q}$ and $\mathcal{P} = [0, \infty)$. In the context of rings it would be interesting whether it is undecidable for commutative polynomials for the cone of sum-of-square polynomials or the non-negative polynomials. This might be the case, as these cones have a richer structure than that of polynomials with nonnegative coefficients. Another open question is the complexity of the positivity problem for orthogonal and unitary matrices. While we have shown that these problems are decidable (\cref{thm:moment_orthogonal} and \cref{cor:moment_unitary}) it is not known whether these problems are for example $\NP$-hard or not.

\section{Acknowledgments} The authors are grateful to Susheel Shankar for pointing out the need for clarification in \cref{rem:generalizations}. This research was funded in part by the Austrian Science Fund (FWF) [doi:\href{https://www.doi.org/10.55776/P33122}{10.55776/P33122}]. AK further acknowledges funding of the Austrian Academy of Sciences (\"OAW) through the DOC scholarship 26547.

\end{document}